\documentclass[runningheads]{llncs_edit}
\usepackage[english]{babel}
\usepackage[dvipsnames]{xcolor}
\usepackage{graphicx,pifont}
\usepackage[a4paper,top=3cm,bottom=3cm,
           left=3cm,right=3cm,marginparwidth=60pt,includehead]{geometry}

\usepackage[T1]{fontenc}
\usepackage[utf8]{inputenc}

\usepackage{amsthm,amsmath,amssymb}
\usepackage{orcidlink}
\usepackage{todonotes}
\usepackage{bm}
\usepackage{bbm}
\usepackage{comment}
\usepackage{soul}

\usepackage[numbers, sort]{natbib}

\usepackage[advantage, adversary, landau, lambda, probability, notions, oracles, asymptotics, noamsfonts]{cryptocode}

\usepackage[capitalise,noabbrev]{cleveref}

\newcommand{\Gr}{{\mathrm{Gr}}}
\newcommand{\Grkn}{{\mathrm{Gr}_q(k,n)}}
\newcommand{\dia}{{\mathcal{D}}}
\newcommand{\diagn}{{\mathcal{D}_{n,q}}}

\newcommand{\Monn}{{\mathcal{M}_{n,q}}}
\newcommand{\K}{\mathbb{K}}

\newtheorem*{maintheorem}{Main Result}

\begin{document}
\title{Linear Code Equivalence via Pl\"ucker Coordinates}
%
%
\author{
    Gessica Alecci \orcidlink{0000-0002-9514-5987} \and 
    Giuseppe D'Alconzo \orcidlink{0000-0001-7377-6617}
}
\institute{
    Politecnico di Torino, Torino, Italy \\
     \email{gessica.alecci@polito.it, giuseppe.dalconzo@polito.it}
}
\authorrunning{G. Alecci and G. D'Alconzo}
\maketitle              

\begin{abstract}
   The assumed hardness of the Linear Code Equivalence problem (LCE) lies at the core of the security of the LESS signature scheme and other signature schemes with advanced functionalities. The LCE problem asks to determine whether two linear codes are equivalent.
   This equivalence is represented by a monomial matrix $\bm Q$, i.e. the product of a diagonal matrix $\bm D$ and a permutation matrix $\bm P$. The recovery of $\bm Q=\bm D \bm P$ is known to be reduced to the recovery of the permutation matrix $\bm P$ alone.
   Exploiting this fact, we construct an algebraic model for LCE involving only the matrix $\bm P$. To this end, we study the action of monomial matrices on linear codes using tools from algebraic geometry, including Pl\"ucker coordinates and fields of invariant rational functions. In particular, we analyse the action of diagonal matrices on linear codes, which can be interpreted as diagonal scaling of the coordinates of elements of the Grassmannian.
   We propose a method to determine algebraically independent generators of the field of rational functions invariant under this action, without relying on Reynolds operators or Gr\"obner basis computations. Furthermore, given two equivalent codes, we apply our results to explicitly construct, for each invariant function, a polynomial having $\bm P$ as a root. However, the resulting polynomials are not of practical use: their degrees are high for cryptographically relevant parameters, and the number of monomials grows exponentially, making them infeasible to manipulate. Despite this limitation, our results are of theoretical interest, as they constitute the first application of these tools to the cryptanalysis of LCE and provide insight into how algebraic geometry and invariant theory can be employed in Cryptography.
    
\end{abstract}

\section{Introduction}\label{sec-introduction}

\subsubsection{Linear Code Equivalence Problem.}
In the realm of Post-Quantum cryptographic assumptions, the Linear Code Equivalence problem (LCE) has received considerable attention in recent years. The problem consists in determining, if it exists, a linear map which preserves the Hamming distance (i.e. a linear isometry) between two linear codes, that is, between two $k$-dimensional vector subspaces of $\mathbb{F}_q^n$. The group of isometries, denoted by $\Monn$, is the set of \emph{monomial matrices}, so the $n\times n$ matrices with exactly one non-zero entry in each row and column. It is known that $\Monn$ is the semidirect product between $\diagn$, the group of invertible diagonal $n\times n$ matrices over $\mathbb{F}_q$, and $\mathcal{S}_n$, the permutation group on $n$ elements. The LCE problem can be naturally formulated as an instance of the following group action problem: given an action $\star$ of a group $G$ over a set $X$, and given two set elements $x,y$, find, if any, a group element $g$ such that $y=g\star x$. In the literature, this is known as \emph{Group Action Inversion Problem} \cite{stolbunov2012} or \emph{Vectorization Problem} \cite{EPRINT:Couveignes06}. A group action for which this problem is intractable is said \emph{one-way} \cite{AC:ADMP20}. In the case of LCE, $X$ is given by the set of linear codes of dimension $k$ and length $n$ (which is the Grassmannian of $k$-dimensional subspaces of $\mathbb{F}_q^n$, denoted by $\Grkn$), while the group $G$ is formed by the isometries $\Monn$.

The first appearance of the Linear Code Equivalence problem in Cryptography is its use in the LESS signature scheme \cite{AFRICACRYPT:BMPS20} and follow-up works \cite{barenghi2022advanced,PQCRYPTO:BBPS21,DCC:ChoPerSan25,AC:PerSan23}. The problem became relevant with the submission of LESS to the NIST's call for post-quantum digital signatures \cite{NIST2023} and its selection for the round-two phase of the call. Since then, the cryptanalysis of LCE intensified, ranging from codeword-search approaches \cite{barenghi2023computational,SAC:Beullens20,EPRINT:BEFN25}, canonical form-based ones \cite{11269862,DCC:ChoPerSan25,PQCRYPTO:Nowakowski25,AC:PerSan23} and ad-hoc strategies for certain classes of codes \cite{EPRINT:BatMorSan25}. Another promising approach is the one represented by algebraic attacks. An algebraic model for a computational problem is a system of polynomials having as a root a solution of the problem. Algebraic attacks to LCE have been studied in \cite{saeed2017algebraic}.

\subsubsection{Quotient Group Actions.}
A significant turning point is represented by the seminal work of Chou, Persichetti, and Santini \cite{DCC:ChoPerSan25}, in which the authors reduce the LCE problem to an equivalent one, where the search for the isometry is restricted to a subset of $\Monn$. This approach is generalized in \cite{d2026group}, showing that if $G=G_1 \rtimes G_2$, then the existence of an efficient way of computing representatives (i.e. a \emph{canonical form}) for the equivalence classes of $X/G_1$ implies that one can use the action of the subgroup $G_2$ on $X/G_1$ and the inversion problem is equivalent to the original one. In \cite{d2026group}, the authors propose an application to the Linear Code Equivalence, setting $\diagn$ as $G_1$ and $\mathcal{S}_n$ as $G_2$, and they introduce a canonical form for the action of $\diagn$ on linear codes. The work \cite{cryptoeprint:2025/397} follows the same approach and provides another canonical form for the diagonal action. One of the strengths of the two works is to reduce the group action inversion problem, where the group is $\Monn$ acting on $\Grkn,$ to the one in which $\mathcal{S}_n$ acts on the orbits $\Grkn/\diagn$.

Current LCE cryptanalysis does not exploit this formulation, but it may be of particular interest for the following reason. Suppose that the group element $g$ sends $x$ to $y$, then its inverse maps $y$ to $x$. Therefore, since we can consider two instances of the problem, if the inverse of $g$ can be written in terms of $g$ without adding new variables, using an algebraic model, we obtain twice as many equations in the same number of unknowns. From the LCE formulation of \cite{d2026group,cryptoeprint:2025/397}, we obtain that the group element is a permutation, and its inverse, when seen as a matrix $\bm P$, is the transpose $\bm P^{-1}=\bm P^{t}$, in particular the entries of $\bm P^{-1}$ are linear in the entries of $\bm P$.

\subsubsection{Multi-samples LCE.}
Advanced cryptographic functionalities could require additional assumptions besides the one-wayness of the action. These other assumptions can be summarized as the $t$-samples Group Action Inversion problem: giving $(x_1,\dots,x_t)$ and $(y_1,\dots,y_t)$ such that $y_i = g\star x_i$, find $g$. For instance, the linkable ring signatures from \cite{barenghi2022advanced,AFRICACRYPT:CNPRRS23} are based on the 2-sample problems on LCE and Matrix Code Equivalence, respectively. As shown in \cite{FFA:DalDis24}, weak unpredictability and weak pseudo-randomness from \cite{AC:ADMP20} are reducible to the \emph{multiple one-wayness}, the assumption saying that, given a polynomial number of samples, it is intractable to find the group element $g$. In \cite{FFA:DalDis24}, it is shown that the textbook formulation of LCE\footnote{The authors consider the action of $\mathrm{GL}_{k,q}\times \Monn$ on $\mathbb{F}_q^{k \times n}$ given by $((\bm S,\bm Q),\bm G) \mapsto \bm S \bm G \bm Q$.} is not multiple one-way. Later, the same result for the actual LCE used in applications, where $\Monn$ acts on $\Grkn$, has been proven in \cite{AC:BCDSK24}. Moreover, in \cite{AC:BCDSK24}, the authors prove that the $2$-sample LCE is heuristically solvable in polynomial time. In the single sample scenario, the different formulations of LCE are equivalent, but this is not true when we deal with multiple samples. Hence, it is not known if the action of $\mathcal{S}_n$ on the orbits $\Grkn/\diagn$ from \cite{d2026group,cryptoeprint:2025/397} is weak pseudo-random, weak unpredictable or multiple one-way.

\subsubsection{Our contribution.}
Writing an algebraic model for the Group Action Inversion Problem of the action of $\mathcal{S}_n$ on the orbits $\Grkn/\diagn$ helps the cryptanalysis of both one-wayness and multiple one-wayness of LCE. The presence of permutations as group elements leads to the possibility of doubling the number of equations using the fact that the entries in $\bm P^{-1}$ are linear in the entries of $\bm P$. For this reason, we tried to answer the following question
\begin{center}
    \textit{Can we find an algebraic model for the Linear Code Equivalence problem formulation with the action of $\mathcal{S}_n$ on the orbits $\Grkn/\diagn$?}
\end{center}
Let $\star$ be the action of $\Monn$ on $\Grkn$. The subgroups $\diagn$ and $\mathcal{S}_n$ of $\Monn$ inherit the same action. Hence, the goal of this work is to model the group action inversion problem of the following action
\begin{equation}\label{star}
    \begin{array}{cccc}
    \star_{\mathcal{S}_n} : \; &\mathcal{S}_n \times \Grkn/\diagn &\to &  \Grkn/\diagn\\
    &(\bm P, [V]_{\diagn}) &\mapsto & [\bm P \star V]_{\diagn},
\end{array}
\end{equation}
where $\bm P$ is a permutation matrix and $[V]_{\diagn}$ is the equivalence class of the vector space $V$ under the action of $\diagn$. Observe that this action is defined as the quotient of the usual one related to the code equivalence problem.

The following theorem partially answers the above question.
\begin{maintheorem}[\Cref{th:main}]\label{th_main}
    Let $\bm G_1$ and $\bm G_2$ be a random instance of the LCE problem, i.e. generator matrices of two equivalent codes over $\mathbb{F}_q$ of dimension $k$ and length $n$. Suppose that they are equivalent via the monomial matrix $\bm Q=\bm P \bm D$. 
    Then, with overwhelming probability over the choice of $\bm G_1$ and $\bm G_2$, there exists a polynomial $h\in \mathbb{F}_q[x_{1,1},\dots,x_{n,n}]$ of degree $2k$ with $2(k!)^2$ monomials which vanishes on the entries of the permutation matrix $\bm P$.
\end{maintheorem}
Thanks to this result, we provide the first algebraic modeling of the LCE problem where the unknowns are the entries of the permutation matrix $\bm P$ leading to the equivalence. Since the degree of the polynomial equations is $2k$, this approach is not practical for parameters sets of cryptographic interest (see, for example, the ones used in \cite{NISTPQC-ADD-R1:LESS23} where $k\ge 126$). Moreover, the polynomials we obtain have a super-polynomial number of monomials, making their use and manipulation not feasible in practice. However, to obtain this algebraic modeling, the action of $\diagn$ on $\Grkn$ has been studied, and an algorithm to explicitly find the generators of the field of invariant functions $\mathbb{F}_q(\Grkn)^{\diagn}$ is presented.

Recall that the computational problem we want to solve using an algebraic model is the following: given representatives of $[\bm P\star V]_{\diagn}$ and $[V]_{\diagn}$, find $\bm P$. To prove the above theorem and write the polynomials of the algebraic model, we take into account the following points.
\begin{enumerate}
    \item The group $\diagn$ acts on $\Grkn$ and we are interested in invariant rational functions, i.e. elements of the field of invariants $\mathbb{F}_q(\Grkn)^{\diagn}$. In fact, given $\mu\in \mathbb{F}_q(\Grkn)^{\diagn}$, we have that $\mu(\lambda \star V) = \mu(V)$ for each $\lambda\in \diagn$ and $V \in \Grkn$. This means that given two vector spaces $U,V$ such that $[U]_{\diagn}=[V]_{\diagn}$, then they satisfy \begin{equation}\label{eq:UV}
        \mu(U)=\mu(V).
    \end{equation}
    \item Consider the action of the permutation group over the quotient $\Grkn/\diagn$. If $\bm Q$ is a monomial matrix given by $\bm Q = \bm D \bm P$, with $\bm P$ permutation matrix and $\bm D$ diagonal, we have that the spaces $\bm Q\star V$ and $\bm P\star V$ are in the same equivalence class with respect to $\diagn$. Hence, \Cref{eq:UV} leads to
    \begin{equation}\label{eq:PV}
        \mu(\bm Q \star V)=\mu(\bm P\star V).
    \end{equation}
    \item Since $\bm Q \star V$ and $V$ are known, treat the entries of the permutation matrix $\bm P$ in \Cref{eq:PV} as unknowns. In this way, for each $\mu$, we find a polynomial in $n^2$ unknowns, and the number of different $\mu$'s we can employ is the number of algebraically independent invariants in $\mathbb{F}_q(\Grkn)^{\diagn}$, i.e. the number of generators of $\mathbb{F}_q(\Grkn)^{\diagn}$ over $\mathbb{F}_q$.
\end{enumerate}
In order to work with rational functions defined on $\Grkn$, we need to parametrize this set. For this reason we use the Pl\"ucker coordinates and we work on the field of coordinates $\mathbb{F}_q(\Grkn)$. Previously known methods for computing the generators rely on the computation of Gr\"obner bases and the Reynolds invariants. In our case, we show that it can be avoided: we present an algorithm which exploits how the action of $\diagn$ modifies the Pl\"ucker coordinates of a space, finding the invariant functions and selecting algebraically independent ones using a variant of the Jacobian criterion.

This article is organized as follows: in \Cref{Sect:prel} we present Algebraic Geometry and Cryptography preliminaries needed to develop and present our results. We then study the action of the group $\diagn$ on $\Grkn$ and provide an algorithm to compute generators of $\mathbb{F}_q(\Grkn)^{\diagn}$ in \Cref{sec:pluck_diag}. An example is given in Subsection \ref{sect:example24}.
Finally, we apply the developed tools to the Linear Code Equivalence problem in \Cref{Sect:LCE}, showing how to obtain an algebraic model in the entries of the permutation matrix describing the equivalence between two codes. An additional example of this technique is provided in Subsection \ref{subs:example_poly}.

\section{Preliminaries}\label{Sect:prel}
\subsection{Notation}
Let $q$ be a prime power and $n$ and $k$ be positive integers such that $k\le n$. The finite field with $q$ elements is denoted by $\mathbb{F}_q$ and the $n$-dimensional vector space over it by  $\mathbb{F}_q^n$. Given a matrix $\bm A \in R^{k \times n}$, the left kernel of $\bm A$ is the set $\{ v \in R^{1\times k} \mid v\bm A = 0\}$.
The set $\{1,\dots,n\}$ is denoted by $[n]$ and the set of its subsets of cardinality $k$ is $\binom{[n]}{k}$. For the rest of the paper we assume the standard lexicographic order on subsets $I\in\binom{[n]}{k}$, hence, there is an implicit bijection between the integers $1,\dots,\binom{n}{k}$ and the elements of $\binom{[n]}{k}$. The union of two sets $A$ and $B$ as multiset is denoted by $A\uplus B$, for example $\{1,2\} \uplus \{2,3\} = \{1,2,2,3\}$.

\subsection{Algebraic Geometry}\label{subs:alggeo}

\subsubsection{Grassmannian and Pl\"ucker Coordinates.}
Let $V$ be an $n$-dimensional vector space over a field $\K$.
The \emph{Grassmannian} $\mathrm{Gr}(k,V)$, also denoted $\mathrm{Gr}(k,n)$ when $V=\K^n$, is the set of all $k$-dimensional vector subspaces of $V$. Each point of $\mathrm{Gr}(k,V)$ corresponds to a $k$-dimensional subspace $W \subset V$.
The Grassmannian $\mathrm{Gr}(k,n)$ is a smooth projective variety of dimension $k(n-k)$.
When $\K = \mathbb{F}_q$ is a finite field with $q$ elements, the Grassmannian is a finite algebraic variety whose $\mathbb{F}_q$-rational points correspond to all $k$-dimensional subspaces of the vector space $\mathbb{F}_q^n$. In this last case, we denote the Grassmannian through $\Gr_q(k,n)$.

If $\{v_1, \dots, v_k\}$ is a basis of $W$, we can represent $W$ by the $k \times n$ matrix $\bm A$ whose rows are the coordinates of the $v_i$.  
Changing the basis of $W$ multiplies this matrix on the left by an invertible $k \times k$ matrix, so the subspace $W$ is determined up to the natural action of $\mathrm{GL}_k$.
For each multi-index \(I = \{i_1 < \cdots < i_k\} \subset \{1,\dots,n\}\) of size \(k\), let
\[
   p_I(W) = \det(\bm A_{I}),
\]
where \(\bm A_{I}\) denotes the \(k \times k\) submatrix of \(\bm A\) formed by the columns indexed by \(I\).  
The collection of the minors
\[
    \left(p_{I_1}(W):\dots: p_{I_{\binom{n}{k}}} (W) \right) \in \mathbb{P}^{\binom{n}{k}-1},
\]
are called the \emph{Plücker coordinates} of the subspace \(W\).  
Observe that each $p_I(W)$ is a homogeneous polynomial of degree $k$ in the entries of $\bm A$.

The map that associates to each subspace $W \subset \K^n$ its Plücker coordinates is called the \emph{Plücker embedding}, defined as
\begin{equation}\label{PlucEmb}
    \begin{array}{cccc}
   \iota : \; &\Gr(k,n) & \hookrightarrow & \mathbb{P}^{\binom{n}{k}-1} \\
    &W & \longmapsto & \left(p_{I_1}(W):\dots:  p_{I_{\binom{n}{k}}} (W) \right).
\end{array}
\end{equation} 

The image of this embedding is not the whole projective space, but rather the subset of points whose Plücker coordinates $p_I$ satisfy certain quadratic relations called  \emph{Plücker relations} \cite{gatto2005schubert,kleiman1972schubert}; given $I\in \binom{[n]}{k-1}$ and $J=\{j_1,\dots,j_{k+1}\}\in \binom{[n]}{k+1}$, one relation is given by
\begin{equation}\label{eq:qr}
    \sum_{\ell =1}^{k+1} (-1)^\ell p_{I\cup \{j_\ell\}}p_{J\setminus \{j_\ell\}} = 0,
\end{equation}
with the common practice that $p_{Y \cup \{y\}}=0$ if $y\in Y$.

\subsubsection{Rational functions.}
Let $\K$ be a field. The \emph{field of rational functions} in $n$ variables over $\K$,
denoted by $\K(x_1,\dots,x_n)$, is the set 
\[
\K(x_1,\dots,x_n)= \mathrm{Frac}(\K[x_1,\dots,x_n]) = \left\{ \frac{f}{g} \;\Bigg|\; f, g \in \K[x_1,\dots,x_n], g \neq 0 \right\}.
\] 
Given an irreducible affine algebraic variety $X \subset \K^n$, its ideal in $\K[x_1,\dots,x_n]$ is $I(X)=\{ F \in \K[x_1,\dots,x_n] \mid  F(P) = 0 \,\text{ for all } P \in X\}$. The ring of coordinates of $X$ is defined as $\K[X]=\K[x_1,\dots,x_n]/I(X)$.
Its fraction field is the field of coordinates of $X$, denoted with $\K(X)$.

Let $G$ be a group acting on the variety $X$ via the map $\star:G\times X\to X$. The action can be extended to polynomials $\K[X]$
by letting $g \in G$ send
$f \in \K[X]$ to the polynomial $g \circ f = f(g^{-1} \star x)$. Then the \emph{ring of $G$-invariant functions on $X$} is given by
\[
    \K[X]^G = \{ f \in \K[X] \mid g \circ f = f \quad \forall g \in G \},
\]
i.e. for each point $x \in X$ and $g \in G$, we have that $f(g^{-1} \star x)= f(x)$. The field of $G$-invariant functions is defined as the fraction field of $\K[X]^G$. The \emph{Reynolds operator} \cite[Sect.~2.1]{sturmfels2008algorithms} of a finite group $G$ acting on a variety $X$ sends a rational function $f$ to an invariant rational function $R(f)$, acting as follows
\begin{equation}
    \begin{array}{cccc}
       R :  &  \K(X) &\to & \K(X)^G\\
         & f & \mapsto & \frac{1}{\lvert G \rvert} \sum_{g \in G} g \circ f .
    \end{array}
\end{equation}
Moreover, since the computation of the Reynolds operator \cite[Sect.~2.1]{sturmfels2008algorithms} has a cost which is linear in the cardinality of the group, it is not doable for groups like $\dia_n$

\subsubsection{Algebraic independence.}
Let $\mathbb{E}$ be an extension of a field $\mathbb{K}$. A set of elements $f_1,\dots,f_\ell\in \mathbb{E}$ is said to be \emph{algebraically independent} over $\mathbb{K}$ if, for any polynomial $P\in \mathbb{K}[y_1,\dots,y_\ell]$, we have $P(f_1,\dots,f_\ell)\ne 0$ in $\mathbb{E}$. The concept of algebraic independence is linked to the linear independence of the differentials. Concerning rational functions, given $f \in \K(x_1,\dots,x_n)$, its partial derivative $\partial f/\partial x_i $ is the formal derivative with respect to $x_i$, while the \emph{differential of $f$} is given by the vector 
\[
    \mathrm{d}f = (\partial f/\partial x_1,\dots,\partial f/\partial x_n).
\]
The \emph{module of K\"ahler differentials} $\Omega_{\mathbb{E}/\K}$ is the set $\{\mathrm{d}f \mid f \in \mathbb{E} \}$ and it is a vector space over the field $\mathbb{E}$. Then, as reported in sections 16.5 and 16.6 of \cite{eisenbud2013commutative}, the set of rational functions $f_1,\dots,f_\ell\in \K(X)$ is algebraically independent over $\K$ if the set of differentials $\mathrm{d}f_1,\dots,\mathrm{d}f_\ell\in \Omega_{\K(X)/\K}$ is linearly independent (over $\K(X)$). As showed in \cite{eisenbud2013commutative}, if $X$ is defined by polynomials $F_1,\dots,F_r \in \K(x_1,\dots,x_n)$, we have that 
\[
    \Omega_{\K(X)/\K} \cong \Omega_{\K(x_1,\dots,x_n)/\K}/\langle \mathrm{d}F_1,\dots, \mathrm{d}F_r \rangle.
\]
We can summarize the algebraic independence criterion when $X=\Gr(k,n)$ and $F_1,\dots, F_r$ are the polynomials defining the Pl\"ucker relations.
\begin{proposition}\label{Lem:Jac}
    Let $F_1,\dots, F_r$ be the polynomials defining the Pl\"ucker relations and let $\left\{\mathrm{d}F_{i_1},\dots, \mathrm{d}F_{i_t}\right\}$ be a maximal subset of linearly independent elements in $\{\mathrm{d}F_1,\dots, \mathrm{d}F_r \}$. The polynomials $f_1\dots,f_\ell$ are algebraically independent over $\K(X)$ if the $(\ell + t) \times n$ matrix whose rows are $\mathrm{d}f_1,\dots, \mathrm{d}f_\ell$, $\mathrm{d}F_{i_1}, \dots, \mathrm{d}F_{i_t}$ has rank $\ell + t$ over $\K(X)$.
\end{proposition}
If the elements $f_1,\dots,f_\ell$ are not algebraically independent, one can use \Cref{Lem:Jac} to choose an algebraic independent subset of them selecting the elements $f_{i_1},\dots,f_{i_s}$ such that the $(s + t) \times n$ matrix whose rows are given by vectors $\mathrm{d}f_{i_1},\dots, \mathrm{d}f_{i_s}, \mathrm{d}F_{i_1}, \dots, \mathrm{d}F_{i_t}$ has rank $s + t$.

\subsection{Cryptography}

\subsubsection{Cryptographic Group Actions.}
Let $(G,\cdot)$ be a group and $X$ be a set, we denote by $\star : G\times X \to X$ an action of $G$ over $X$.
In the cryptographic context, a group action is said \emph{effective} if there are polynomial-time algorithms for computing the group action $\star$, the group operation, the inverses, the canonical bitstring representation and sapling in $G$ and $X$. See \cite{AC:ADMP20} for a more formal definition. We define the following problem.
\begin{definition}\label{def:GAIP}
    The \emph{Group Action Inversion problem} of the action $(G,X,\star)$ asks, given $x,y \in X$, to find, if any, $g \in G$ such that $g \star x=y$.
\end{definition}
In Cryptography, in order to have a suitable group action we want that, given \emph{random} $x\in X$ and $g \in G$, computing $g$ from the pair $(x, g\star x)$ is assumed to be infeasible, i.e. exponential in $\log(\lvert G \rvert)$. An action having this property is said \emph{one-way}. Some cryptographic constructions need more advanced assumptions, like the \emph{multiple one-wayness}: for a polynomial number of pairs $(x_i, g\star x_i)$, computing $g \in G$ must still be intractable (see \cite{FFA:DalDis24} for the formal definition). In \cite{FFA:DalDis24}, the authors show that this assumption do not hold for some actions, even if they are assumed to be one-way. A way to model the Group Action Inversion Problem, given two set elements $x$ and $y$, is to write a system of polynomials having as a root the group element $g$ such that $y=g \star x$. Such a polynomial system is called \emph{algebraic model} for the problem and can be used to perform what is called an algebraic attack.

\subsubsection{Quotient Group Actions.}
Let $G$ be a group such that $G = G_1 \rtimes G_2$ for two subgroups $G_1$ and $G_2$. In particular, $G_1$ is a normal subgroup of $G$. Let $\sim_{G_1}$ be the equivalence relation on $X$ given by
\[
    x \sim_{G_1} y \iff \exists g \in G_1 \text{ s.t. } y = g \star x.
\]
Let $X/G_1 = \{[x]_{G_1} \mid x \in X\}$ be the set of equivalence classes of $X$ under $\sim_{G_1}$. As in \cite{d2026group}, we define the \emph{quotient action} of $G_2$ over the $X/G_1$ as follows
\[
    \star_{G_2} : G_2 \times X/G_1 \to X/G_1, \; (g,[x]_{G_1} ) \mapsto [g \star x]_{G_1}.
\]
More generally, a \emph{canonical form} $\mathsf{CF}: X \to X$ for the equivalence $\sim$ is a map such that
\begin{enumerate}
    \item $x \sim \mathsf{CF}(x)$;
    \item $\mathsf{CF}(x) = \mathsf{CF}(y)$ if and only if $x \sim y$.
\end{enumerate}
The existence of a polynomial-time computable canonical form for the equivalence $\sim_{G_1}$, and hence for the classes of $X/G_1$, implies that if $\star$ is effective then $\star_{G_2}$ is effective too \cite{d2026group}. Under some assumptions, the security estimates of the one-wayness of original action $\star$ transpose to the quotient one $\star_{G_2}$, and the latter can be used in cryptographic protocols replacing the former, bringing some advantages \cite{d2026group}. On the cryptanalytic side, analyzing the one-wayness of the quotient action is equivalent to analyzing the one-wayness of the starting one. Counter-intuitively, the same is not true for the multiple one-wayness; it can happen that even if $\star$ is not multiple one-way, then $\star_{G_2}$ can achieve this property.

\subsubsection{Code Equivalence.}
A well-known cryptographic group action studied in the literature is the one associated with the Linear Code Equivalence problem. This problem asks to retrieve the isometry between two equivalent codes, and can be embedded in the group action framework. A linear code over $\mathbb{F}_q$ of length $n$ and dimension $k$ is a vector subspace of dimension $k$ of $\mathbb{F}_q^n$ (and hence, an element of $\Grkn$) endowed with the \emph{Hamming distance}, a map counting the number of different entries between two vectors. A linear code $\mathcal{C}$ in $\Grkn$ can be represented by a generator matrix $\bm G$, that is an element of $\mathbb{F}_q^{k\times n}$ whose rows form a basis of the code. We say that $\bm G$ generates (or spans) the code $\mathcal{C}$. Of particular interest are the linear map which preserve the Hamming weight, since they give a way of classify linearly equivalent codes. Such maps are called linear isometries and are given by monomial matrices $\Monn$, i.e. product of non-singular diagonal matrices and permutation matrices. Hence, $\Monn$ acts on a linear code of length $n$ and dimension $k$ permuting and scaling its coordinates. This action can be represented as an action on its generator matrix. In order for the Group Action Inversion Problem to be intractable and thus of cryptographic interest, the action we consider requires a change of basis given by an invertible $k \times k$ matrix in $\mathrm{GL}_{k,q}$. In this case, the set $X$ of linear codes became the set $\mathbb{F}_q^{k\times n,k}$ of $k \times n$ matrices of rank $k$ and the group $G$ is the direct product $\mathrm{GL}_{k,q}\times \Monn$, leading to the action given by
\[
    \begin{array}{cccc}
   \star : & \left( \mathrm{GL}_{k,q}\times \Monn \right) \times \mathbb{F}_q^{k\times n,k} & \to & \mathbb{F}_q^{k\times n,k}, \\
    &((\bm S, \bm Q), \bm M) &\mapsto & \bm S^{-1} \bm M \bm Q.
\end{array}
\]
\begin{definition}
    The Linear Code Equivalence (LCE) problem with parameters $q,k,n$ asks, given $\bm G_1,\bm G_2 \in \mathbb{F}_q^{k\times n,k}$, to find $(\bm S,\bm Q) \in \mathrm{GL}_q(k)\times \Monn$ such that $\bm S \bm G_1 \bm Q = \bm G_2$.
\end{definition}
Even if one assumes that the LCE problem is intractable, and hence that the action is one-way, in \cite{FFA:DalDis24} it is shown that it is not multiple one-way.

In practice, since the size of the elements of the acting group is a crucial metric in applications like digital signatures, one can get rid of the invertible matrix $\bm S$ using the reduced row echelon form (RREF) for the matrix $\bm M \bm Q$. The reduced row echelon form can be seen as a canonical form for the classes of $\mathbb{F}_q^{k\times n,k}$ under the relation induced by the action of the subgroup $\mathrm{GL}_{q,k}$, falling into the framework from \cite{d2026group}. Hence, in the cryptographic literature, the Linear Code Equivalence is presented as the Group Action Inversion problem (\Cref{def:GAIP}) of the following group action
\begin{equation}
    \begin{array}{cccc}
    \star : \; &\Monn \times \mathbb{F}_q^{k\times n,k} &\to &  \mathbb{F}_q^{k\times n,k},\\
    &(\bm Q, \bm M) &\mapsto & \mathrm{RREF}(\bm M \bm Q).
\end{array}
\end{equation}
As shown in \cite{AC:BCDSK24,budroni2025sample}, even this action is not multiple one-way.

The group $\Monn$ can be further factorized as $\diagn \rtimes \mathcal{S}_n$, and a quotient action of $\mathcal{S}_n$ can be defined on the classes of the relation induced by $\diagn$, as done in \cite{d2026group}\footnote{Further factorizations of $\Monn$ can be exploited to reduce the sizes of the group elements, as in \cite{AC:PerSan23,DCC:ChoPerSan25} with the drawback that they do not directly lead to a group action.}. Observing that the set of linear codes is exactly the Grassmannian $\Grkn$, we can treat this new action as
\begin{equation}
    \begin{array}{cccc}
    \star_{\mathcal{S}_n} : \; &\mathcal{S}_n \times \Grkn/\diagn &\to &  \Grkn/\diagn\\
    &(\bm P, [V]_{\diagn}) &\mapsto & [\bm P \star V]_{\diagn}.
\end{array}
\end{equation}
In \cite{d2026group,cryptoeprint:2025/397}, the authors provide a polynomial-time canonical form for the classes of $\Grkn$ under the action of $\diagn$, showing that $\star_{\mathcal{S}_n}$ is an effective group action. Even if the one-wayness of $\star_{\mathcal{S}_n}$ is directly linked to the one of the starting Linear Code Equivalence action, its multiple one-wayness has not been studied. Hence, plugging the former in protocols which require multiple one-wayness, previously instantiated with the latter, is not secure a priori.

\section{Pl\"ucker Coordinates and the Diagonal Action}
\label{sec:pluck_diag}

We want to study the action of $\diagn$ on the Grassmannian and how it translates on its image via the Pl\"ucker embedding $\iota:\Grkn \hookrightarrow \mathbb{P}^{\binom{n}{k}-1}$. However, our results remain valid for any field $\K$. We therefore consider the diagonal group $\dia_n$ of non-singular diagonal matrices over $\K$.

First, let us consider the following action of the diagonal group $\dia_n$ over the whole projective space $\mathbb{P}^{\binom{n}{k}-1}$
\begin{equation}\label{startilde}
    \begin{array}{cccc}
    \tilde\star : \; &\dia_n \times \mathbb{P}^{\binom{n}{k}-1} &\to &  \mathbb{P}^{\binom{n}{k}-1}\\
    &\left(\left(\lambda_1,\dots,\lambda_n\right),\left(x_{I_1}:\dots:  x_{I_{\binom{n}{k}}} \right)\right) &\mapsto & \left( \left(\prod_{i\in I_1} \lambda_i\right) x_{I_1}:\cdots:\left(\prod_{i\in I_{\binom{n}{k}}} \lambda_i\right)x_{I_{\binom{n}{k}}}\right).
    \end{array}
\end{equation}
The above action behaves exactly as the image of the action of $\dia_n$ on $\Gr(k,n)$ under the Pl\"ucker embedding.
\begin{proposition}

    The action $\star$ of $\dia_n$ over the Grassamannian $\Gr(k,n)$ naturally induces an action of $\dia_n$ over the image $\iota(\Gr(k,n))$ of the Pl\"ucker embedding. The latter coincides with $\tilde\star$ restricted to $\iota(\Gr(k,n))$.
\end{proposition}
\begin{proof}
    We want to show that for each $V$ in $\Gr(k,n)$ and each $\lambda\in \dia_n$, we have $\lambda \tilde\star \iota(V)=\iota(\lambda \star V)$. We know that if $\iota(V) = \left(p_{I_1}(V) : \cdots : p_{I_{\binom{n}{k}}}(V) \right)$, then
    \[
        \lambda \tilde\star \iota(V) = \left( \left(\prod_{i\in I_1} \lambda_i \right) p_{I_1}(V) :\cdots: \left(\prod_{i\in I_{\binom{n}{k}}} \lambda_i\right)p_{I_{\binom{n}{k}}}(V) \right).
    \]
    On the other hand, consider $\iota(\lambda \star V)$. To compute it, let $\bm A$ be a matrix whose rows form a basis of $V$. Then define the matrix \( \tilde {\bm A} \) whose \( i \)-th column is obtained by scaling
the \( i \)-th column of \( \bm A \) by \( \lambda_i \). The rows of $\tilde {\bm A}$ form a basis of $\lambda \star V$. By definition, the $I$-th coordinate $p_I(\lambda \star V)$ of $\iota(\lambda \star V)$ is the determinant of $\tilde {\bm A}_I$, the $k\times k$ submatrix of $\tilde {\bm A}$ indexed by the set $I$. This is exactly $\left( \prod_{i\in I}\lambda_i\right) p_I(V)$, and hence $\lambda \tilde\star \iota(V)=\iota(\lambda \star V)$.
\end{proof}

We aim to define non-trivial invariants of the orbit, that is, non-constant elements of the field of rational invariants $\mathbb{K}\left(\mathrm{Gr}(k,n)\right)^{\dia_n}$. As shown in \Cref{sec-introduction}, each invariant leads a polynomial equation for an algebraic model for LCE; for this reason we are interested in finding the algebraically independent ones, i.e. the generators of the field $\mathbb{K}\left(\mathrm{Gr}(k,n)\right)^{\dia_n}$. The minimal number of such generators is given by the transcendence degree of the field $\mathbb{K}\left(\mathrm{Gr}(k,n)\right)^{\dia_n}$ over $\K$. In general, for finite groups a minimal set of generators can be computed explicitly as shown in \cite[Chapt.~2]{sturmfels2008algorithms}. However, this could be infeasible in practice, as it involves the computation of Gr\"obner bases and Reynolds operators. Moreover, since the computation of the Reynolds operator \cite[Sect.~2.1]{sturmfels2008algorithms} has a cost which is linear in the cardinality of the group, it is not doable for groups like $\dia_n$.

Nevertheless, we propose a strategy to compute a set of independent generators for the field $\mathbb{K}\left(\mathrm{Gr}(k,n)\right)^{\dia_n}$ over $\K$ without relying on Gr\"obner bases or the Reynolds operator.

We first report the following combinatorial result, which is a direct consequence from \cite{wilson1990diagonal}.
\begin{lemma}\label{Lem:W}
    Let $\bm W_{k,n}$ be the matrix in $\mathbb{Z}^{\binom{n}{k}\times n}$ whose rows are indexed by elements of $\binom{[n]}{k}$ and the $I$-th row is $\sum_{i\in I}e_i$, where $e_i\in \{0,1\}^n$ is the $i$-th element of the canonical basis. Then
    \[
        \mathrm{rk}_{\mathbb{Z}} \left(\bm W_{k,n} \right) = 
        \begin{cases}
            0 & \text{ for } k=0;\\
            n & \text{ for } 1 \le k \le n-1;\\
            1 & \text{ for } k=n.
        \end{cases}
    \]
\end{lemma}
\begin{proof}
    For $k\in\{0,n\}$, the thesis easily follows. Let $k\notin \{0,n\}$. The matrix $\bm W_{k,n}$ is the \emph{incidence matrix of $t$-subsets vs. $k$-subsets} \cite{wilson1990diagonal}, with $t=1$. Its rank over $\mathbb{Z}$ is given by
    \[
        \mathrm{rk}_{\mathbb{Z}} \left(\bm W_{k,n} \right) =  \sum_{i=0}^t \binom{n}{i} - \binom{n}{i-1}.
    \]
    Hence, assuming $\binom{n}{-1}=0$, the rank of $\bm W_{k,n}$ is $\binom{n}{0}-\binom{n}{-1} + \binom{n}{1}-\binom{n}{0} =n $.
\end{proof}
Now, we can present our algorithm. The idea is that, if $\bm W_{k,n}$ is the matrix in $\mathbb{Z}^{\binom{n}{k}\times n}$ whose rows are indexed by elements of $\binom{[n]}{k}$ and the $I$-th row is $\sum_{i\in I}e_i$, then the elements of its left kernel can be associated to invariants in $\mathbb{K}\left(\mathrm{Gr}(k,n)\right)^{\dia_n}$. The left kernel $K(\bm W_{k,n})$ of $\bm W_{k,n}$ is a free $\mathbb{Z}$-module of rank $\binom{n}{k}-\mathrm{rk}_{\mathbb{Z}} \left(\bm W_{k,n} \right)$. Applying \Cref{Lem:W} and assuming that $k\notin \{0,n\}$, this rank is equal to $\binom{n}{k}-n$. For each element $v$ in $K(\bm W_{k,n})$, we can define the rational function
\begin{equation}\label{eq:fv}
    f_v := \prod_{I_j\in \binom{[n]}{k}}p_{I_j}^{v_{I_j}} \in \mathbb{K}\left(\mathrm{Gr}(k,n)\right),
\end{equation}
where the $p_{I_j}$ are the Pl\"ucker coordinates, and show that it is in $\mathbb{K}\left(\mathrm{Gr}(k,n)\right)^{\dia_n}$. Moreover, we have that
\begin{equation}\label{eq:kernel_vect}
    f_{v+w} = f_v\cdot f_w \;\text{ and }\; f_{m v} = (f_v)^m \qquad \forall m\in \mathbb{Z},\, \forall v,w\in \mathbb{Z}^{\binom{n}{k}}.
\end{equation}
Hence, we can find a basis $v_1,\dots,v_{\binom{n}{k}-n} \in \mathbb{Z}^{\binom{n}{k}}$ of $K(\bm W_{k,n})$ and compute the rational functions $f_1,\dots,f_{\binom{n}{k}-n}$ associated to these vectors. Since we are interested in a set of algebraic independent generators, we can use the criterion given in \Cref{Lem:Jac} to check whether the functions $f_1,\dots,f_{\binom{n}{k}-n}$ are independent and, if not, select the independent ones.

We can then summarize this procedure in the algorithm $\textsf{InvGen}$, given in \Cref{alg:gen}. We now show its correctness.
\begin{figure}[t]
    \centering
\fbox{%
\procedure[linenumbering,skipfirstln]{$\textsf{InvGen}(k,n,\K)$}
{
    \text{\textbf{Input:} positive integers $k,n$ such that $k\le n$ and a field $\K$.} \pcskipln \\ 
    \text{\textbf{Output:} $f_1,\dots,f_\ell$ alg. ind. generators of $\mathbb{K}\left(\mathrm{Gr}(k,n)\right)^{\dia_n}$. } \\
    \text{Define $\bm W_{k,n} \in \mathbb{Z}^{\binom{n}{k}\times n}$ as in \cref{Lem:W}}\\
    \text{Compute a basis $\{v_1,\dots,v_r\}$ for the left kernel $L\subset \mathbb{Z}^{\binom{n}{k}}$ of $\bm W_{k,n}$}\\
    \pcfor i=1,\dots, r \\
    \pcind \text{Set $g_i$ as $\prod_{I_j\in \binom{[n]}{k}}p_{I_j}^{(v_i)_{I_j}}$}\\
    \text{Use Prop. \ref{Lem:Jac} to determine the alg. ind. $g_{i_1},\dots,g_{i_\ell}$ among the $g_1,\dots,g_r$}\\
    \pcreturn g_{i_1},\dots,g_{i_\ell}
}
}
    \caption{Algorithm $\mathsf{InvGen}$ for the computation of a minimal set of independent generators of $\mathbb{K}\left(\mathrm{Gr}(k,n)\right)^{\dia_n}$.}
    \label{alg:gen}
\end{figure}

\begin{theorem}
    If $g_1,\dots,g_\ell$ is the output of $\mathsf{InvGen}(k,n,\K)$, then
    \[
        \K(g_1,\dots,g_\ell) = \K\left(\mathrm{Gr}(k,n)\right)^{\dia_n}.
    \]
\end{theorem}
\begin{proof}
    For the rest of the proof, let $p$ denotes the unknowns $\left(p_{I_1}:\cdots: p_{I_{\binom{n}{k}}}\right)$. We first show that for every $v\in K(\bm W_{k,n})$, the rational function
\[
    f_v (p)= \prod_{I_j\in \binom{[n]}{k}}p_{I_j}^{v_{I_j}} \in \mathbb{K}\left(\mathrm{Gr}(k,n)\right)
\]
is an element of $\mathbb{K}\left(\mathrm{Gr}(k,n)\right)^{\dia_n}$. In other words, $f_v$ is invariant under the action of $\dia_n$ and satisfies $f_v(\lambda \star p ) = f_v(p)$ for each $\lambda \in \dia_n$.
Hence, we have
    \begin{equation*}\label{eq:f_v}
        f_v(\lambda \star p) = \prod_{I_j\in \binom{[n]}{k}}\left( \prod_{t\in I_j}\lambda_t \right)^{v_{I_j}}p_{I_j}^{v_{I_j}},
    \end{equation*}
    which can be rewritten as
    \begin{equation}\label{eq:f_v2}
       f_v(\lambda \star p) = \left(\prod_{I_j\in \binom{[n]}{k}} \prod_{t\in I_j}\lambda_t^{v_{I_j}} \right)  \left(\prod_{I_j\in \binom{[n]}{k}}p_{I_j}^{v_{I_j}}\right).
    \end{equation}
    Consider the first factor in \Cref{eq:f_v2}. For each $I \in \binom{[n]}{k}$, let $\mathbbm{1}_I$ denote the indicator function of $I$, i.e. $\mathbbm{1}_I(t)$ evaluates to $1$ if $t\in I$ and 0 otherwise. Then, we have that
    \begin{equation}\label{eq:f_v3}
        \prod_{I_j\in \binom{[n]}{k}} \prod_{t\in I_j}\lambda_t^{v_{I_j}} = \prod_{I_j\in \binom{[n]}{k}} \prod_{t=1}^n\lambda_t^{v_{I_j} \cdot \mathbbm{1}_{I_j}(t)} = \prod_{t=1}^n\lambda_t^{\sum_{I_j\in \binom{[n]}{k}}v_{I_j} \cdot \mathbbm{1}_{I_j}(t)}.
    \end{equation}
    Observe that the vector $\left(\mathbbm{1}_{I_1}(t), \dots, \mathbbm{1}_{I_{\binom{n}{k}}}(t)\right)^T$ is exactly $W_t$, the $t$-th column of $\bm W_{k,n}$. Hence, \Cref{eq:f_v3} can be written as
    \begin{equation}\label{eq:f_v4}
        \prod_{t=1}^n\lambda_t^{\sum_{I_j\in \binom{[n]}{k}}v_{I_j} \cdot \mathbbm{1}_{I_j}(t)} = \prod_{t=1}^n\lambda_t^{\langle v, W_t \rangle} = 1,
    \end{equation}
    where the last equality holds since $v$ is in the left kernel of $\bm W_{k,n}$. Hence, $f_v$ is invariant under the action of $\dia_n$. If we set $f_i$ as the rational function associated to the $i$-th vector of the base of $K(\bm W_{k,n})$, this completes the inclusion $\mathbb{K}\left(f_1,\dots,f_{\binom{n}{k}-n}\right) \subseteq \mathbb{K}\left(\mathrm{Gr}(k,n)\right)^{\dia_n}$.

    To prove the other inclusion, we proceed in two steps. First, we prove that every monomial in $\mathbb{K}\left(\mathrm{Gr}(k,n)\right)^{\dia_n}$ comes from an element $v$ of the left kernel of $\bm W_{k,n}$ via the correspondence given in \Cref{eq:fv}. Indeed, let $g$ be a monomial in $\mathbb{K}\left(\mathrm{Gr}(k,n)\right)^{\dia_n}$ and without loss of generality we can assume it to be monic. Then, there is an integer vector $a \in \mathbb{Z}^{\binom{n}{k}}$ such that
    \[
        g = \prod_{I_j\in \binom{[n]}{k}} p_{I_j}^{a_{I_j}}.
    \]
    Since it is invariant under the action of $\dia_n$, we have
    \[
        \prod_{I_j\in \binom{[n]}{k}}\left( \prod_{t\in I_j}\lambda_t \right)^{a_{I_j}}p_{I_j}^{a_{I_j}} = \prod_{I_j\in \binom{[n]}{k}} p_{I_j}^{a_{I_j}},
    \]
    and, proceeding as before, we obtain that $a$ is in the right kernel of $\bm W_{k,n}$. To conclude the proof, we show that the sum of monomials that are not in $\mathbb{K}\left(\mathrm{Gr}(k,n)\right)^{\dia_n}$ cannot belong to $\mathbb{K}\left(\mathrm{Gr}(k,n)\right)^{\dia_n}$. To show this, let $g$ and $h$ be monomials in $\mathbb{K}\left(\mathrm{Gr}(k,n)\right) \setminus \mathbb{K}\left(\mathrm{Gr}(k,n)\right)^{\dia_n}$, and, without loss of generality, assume that they are monic. If $K(\bm W_{k,n})$ denotes the left kernel of $\bm W_{k,n}$, this means that $g$ and $h$ are obtained by vectors $a$ and $b$ both in $\mathbb{Z}^{\binom{n}{k}}\setminus K(\bm W_{k,n})$. Hence
    \[
        g = \prod_{I_j\in \binom{[n]}{k}} p_{I_j}^{a_{I_j}} \text{ and } h = \prod_{I_j\in \binom{[n]}{k}} p_{I_j}^{b_{I_j}}.
    \]
    Now, suppose by contradiction that $g+h \in \mathbb{K}\left(\mathrm{Gr}(k,n)\right)^{\dia_n}$. Since $(g+h)(\lambda \star p) = (g+h)(p)$, and by definition $(g+h)(p) = g(p) + h(p)$, if $g+h$ is invariant under the action of $\dia_n$, we have 
    \[
        (g+h)(\lambda \star p) = (g+h)(p).
    \]
    The left hand side can also be written as
    \[
        (g+h)(\lambda \star p) = \prod_{I_j\in \binom{[n]}{k}}\left( \prod_{t\in I_j}\lambda_t \right)^{a_{I_j}}p_{I_j}^{a_{I_j}} + \prod_{I_j\in \binom{[n]}{k}}\left( \prod_{t\in I_j}\lambda_t \right)^{b_{I_j}}p_{I_j}^{b_{I_j}},
    \]
    while the right one as 
    \[
        (g+h)(p) = \prod_{I_j\in \binom{[n]}{k}} p_{I_j}^{a_{I_j}} + \prod_{I_j\in \binom{[n]}{k}} p_{I_j}^{b_{I_j}}.
    \]
    Using the same argument of the first part of the proof, we obtain that both the vectors $a$ and $b$ are in the right kernel of $\bm W_{k,n}$, which is a contradiction. This shows that all polynomials in $\mathbb{K}\left(\mathrm{Gr}(k,n)\right)^{\dia_n}$ are sums of monomials obtained by vectors in $K(\bm W_{k,n})$, and hence, from \Cref{eq:kernel_vect}, all the elements in $\mathbb{K}\left(\mathrm{Gr}(k,n)\right)^{\dia_n}$ arise from $f_1,\dots,f_{\binom{n}{k}-n}$. Therefore, we conclude that
    \[
        \mathbb{K}\left( f_1,\dots,f_{\binom{n}{k}-n} \right) = \mathbb{K}\left(\mathrm{Gr}(k,n)\right)^{\dia_n}.
    \]
    Finally, \Cref{Lem:Jac} proves the correctness of Algorithm \ref{alg:gen}, since, if $g_1,\dots,g_\ell$ are the $\ell$ algebraic independent elements among $f_1,\dots,f_{\binom{n}{k}-n}$, we have that $\mathbb{K}\left(f_1,\dots,f_{\binom{n}{k}-n} \right) = \mathbb{K}(g_1,\dots,g_\ell)$ and hence $\mathbb{K}(g_1,\dots,g_\ell) = \mathbb{K}\left(\mathrm{Gr}(k,n)\right)^{\dia_n}$.
\end{proof}


\begin{remark}\label{rem1}
    The number of such independent invariants is the transcendence degree of $\mathbb{K}\left(\mathrm{Gr}(k,n)\right)^{\dia_n}$ over $\mathbb{K}$. This is given by  \cite[Corollary p. 156]{popov1994invariant}, stating that it is equal to the dimension of $\mathrm{Gr}(k,n)$ minus the dimension of the generic orbit $\mathcal{O}$ of $\dia_n$, namely,
    \[
        \mathrm{tr}\deg_{\mathbb{K}}\left(\mathbb{K}\left(\mathrm{Gr}(k,n)\right)^{\dia_n}\right) = \dim(\Gr(k,n)) - \dim(\mathcal{O}).
    \]
    We know that $\dim(\Grkn)=k(n-k)$, while the dimension of a generic orbit of $\dia_n$ is $n-1$. Hence, the number of independent invariant rational functions is $k(n-k)-n+1$.
\end{remark}

\subsection{The example on $\Gr(2,4)$}\label{sect:example24}
Consider the case $k=2$ and $n=4$. We present an example in which we compute the Pl\"ucker coordinates and invariants under the action of $\dia_4$. We show that for $\Gr(2,4)$ there is only one algebraically independent invariant as stated in \Cref{rem1}, since $k(n-k) - n +1 = 1$.
Hence, let $V \in \Gr(2,4)$ be a vector space generated by the rank-2 matrix
\begin{equation*}
   \bm A = \begin{pmatrix}
a & b & c & d\\
e & f & g & h
\end{pmatrix}.
\end{equation*}
Let $\dia_4= \left\{ (\lambda_1,\lambda_2,\lambda_3,\lambda_4)\in (\mathbb{K}^\times)^4 \right\}$. Then $\lambda \star V$ is generated by
\begin{equation}
\begin{pmatrix}
\lambda_1 a & \lambda_2 b & \lambda_3 c & \lambda_4 d\\
\lambda_1 e & \lambda_2 f & \lambda_3 g & \lambda_4 h
\end{pmatrix}.
\end{equation}
The Pl\"ucker embedding can be computed as
\begin{equation}
    \begin{array}{cccc}
   \iota : \; &\Gr(2,4) & \hookrightarrow & \mathbb{P}^5 \\
     &V & \longmapsto & (p_{12}:p_{13}:p_{14}:p_{23}:p_{24}:p_{34})
\end{array}
\end{equation} 
where $p_{ij} = \det(\bm A_{\{i,j\}})$. The action induced by $\dia_4$ over $\iota (\Gr(2,4))$ is
\begin{align*}
    \tilde\star : \; &\dia_4 \times \iota (\Gr(2,4)) \to \iota (\Gr(2,4))\\ &\left(\left(\lambda_1,\lambda_2,\lambda_3,\lambda_4\right),\left(p_{12}:p_{13}:p_{14}:p_{23}:p_{24}:p_{34} \right)\right) \mapsto  \\
     & \; \left(\lambda_1\lambda_2 p_{12}:\lambda_1\lambda_3 p_{13}: 
\lambda_1\lambda_4 p_{14}: \lambda_2\lambda_3 p_{23}: \lambda_2\lambda_4 p_{24}: \lambda_3\lambda_4 p_{34}\right).
\end{align*}
Now we want to study the field of rational functions that are invariant under $\dia_4$, namely $\mathbb{K}\left(\Gr(2,4)\right)^{\dia_4}$ and find a transcendental basis over $\mathbb{K}$. Let
\begin{equation*}
    \bm W_{2,4} =
\begin{pmatrix}
1 & 1 & 0 & 0\\
1 & 0 & 1 & 0\\
1 & 0 & 0 & 1\\
0 & 1 & 1 & 0\\
0 & 1 & 0 & 1\\
0 & 0 & 1 & 1
\end{pmatrix}
\in \mathbb{Z}^{6\times 4}.
\end{equation*}
A basis of the left kernel of $\bm W_{2,4}$ is given by $\{v_1, v_2 \}$ with $v_1=(1,0,-1,-1,0,1)$ and $v_2=(0,1,-1,-1,1,0)$. From $\mu_{v_i}=\prod_{I_j\in \binom{[4]}{2}}p_{I_j}^{(v_i)_{I_j}}$, it follows
\[
    \mu_{v_1}=\frac{p_{12}p_{34}}{p_{14}p_{23}} \;\text{ and }\; \mu_{v_2}=\frac{p_{13}p_{24}}{p_{14}p_{23}}.
\]

In order to find the algebraically independent invariants, i.e. the minimal set of generators of $\mathbb{K}\left(\Gr(2,4)\right)^{\dia_4}$, we can use the criterion from \Cref{Lem:Jac}.
In particular, let
\[
F = p_{12}p_{34} - p_{13}p_{24} + p_{14}p_{23},
\]
to be the polynomial of the Pl\"ucker relation, then the differentials of $\mu_{v_1}, \mu_{v_2}$ and $F$ are given by
\[
\mathrm{d}\mu_{v_1} =
\left(
\frac{p_{34}}{p_{14}p_{23}},
0,
-\frac{p_{12}p_{34}}{p_{14}^2p_{23}},
-\frac{p_{12}p_{34}}{p_{14}p_{23}^2},
0,
\frac{p_{12}}{p_{14}p_{23}}
\right),
\]
\[
\mathrm{d}\mu_{v_2} =
\left(
0,
\frac{p_{24}}{p_{14}p_{23}},
-\frac{p_{13}p_{24}}{p_{14}^2p_{23}},
-\frac{p_{13}p_{24}}{p_{14}p_{23}^2},
\frac{p_{13}}{p_{14}p_{23}},
0
\right),
\]
and
\[
\mathrm{d}F = (p_{34}, -p_{24}, p_{23}, p_{14}, -p_{13}, p_{12}).
\]
It can be seen that the matrix whose rows are \(\mathrm{d}F\), \(\mathrm{d}\mu_{v_1}\), and \(\mathrm{d}\mu_{v_2}\) has rank \(2\) over the field $\mathbb{F}_4(p_{12},p_{13},p_{14},p_{23},p_{24},p_{34})$. Since $\mathrm{d}F$ is non-zero, then $\mu_{v_1}$ and $\mu_{v_2}$ are not algebraically independent.

However, in the case of $\Gr(2,4)$, one can directly apply the single Pl\"ucker relation $p_{12}p_{34}-p_{13}p_{24}+p_{14}p_{23}=0$. From that, it follows that
\[
    \mu_{v_2}= \frac{p_{13}p_{24}}{p_{14}p_{23}} = \frac{p_{12}p_{34} + p_{14}p_{23}}{p_{14}p_{23}} = \mu_{v_1}+1.
\]
This confirms that, as expected, in this case there is only one algebraically independent invariant.

\section{An Algebraic Geometric approach to solve LCE}\label{Sect:LCE}
We now specialise the techniques developed in \Cref{sec:pluck_diag} in the case of $\K=\mathbb{F}_q$. Let $\Grkn$ be the set of $k$-dimensional vector subspaces (i.e. linear codes) of $\mathbb{F}_q^n$. Recall the quotient group action
\[
\begin{array}{cccc}
    \star_{\mathcal{S}_n} : \; &\mathcal{S}_n \times \Grkn/\diagn &\to &  \Grkn/\diagn\\
    &(\bm P, [V]_{\diagn}) &\mapsto & [\bm P \star V]_{\diagn}.
\end{array}
\]
Since linear codes are represented by their generator matrices, we may, without loss of generality, consider the action on matrices rather than on vector spaces. Hence, in the following, $\bm P \star_{\mathcal{S}_n} [\bm G]_\diagn$ denotes a generator matrix of a code in the equivalence class $\bm P \star_{\mathcal{S}_n} [C]_\diagn$, where $C$ is the code generated by $\bm G$. For this reason we will omit the change of basis matrix $\bm S\in \mathrm{GL}_{k,q}$.

We assume that we can compute in polynomial time a representative for each class, using a canonical form $\mathsf{CF}$, see for instance \cite{d2026group,cryptoeprint:2025/397}. Hence, we can define the action $\star_{\mathcal{S}_n}$ on this set of representatives. Then, given a generator matrix $\bm G$ and a permutation matrix $\bm P$, we have
\[
\bm P \star_{\mathcal{S}_n} \mathsf{CF}(\bm G) := \mathsf{CF}(\bm G \bm P).
\]
Therefore, the Linear Code Equivalence problem reduces to the following one: given two matrices $\bm G$ and $\bm G'$, we want to find the permutation matrix $\bm P$ such that
\[
    \mathsf{CF}(\bm G') = \mathsf{CF}(\bm G \bm P).
\]
This is equivalent to find a permutation matrix $\bm P$ such that $\bm G \bm P$ is in the equivalence class of $\bm G'$ under the action of $\diagn$.

The goal of this section is to explicitly write down an algebraic modeling for this problem, i.e. a polynomial system over $\mathbb{F}_q(\mathcal{S}_n)$. The advantages of having a system which unknown is a permutation is given by the following fact. If $\mathsf{CF}(\bm G_2) = \mathsf{CF}(\bm G_1 \bm P)$ then we have that $\mathsf{CF}(\bm G_1) = \mathsf{CF}(\bm G_2 \bm P^{-1})$. Observing that $\bm P^{-1} = \bm P^T$, we get other equations in the same unknown of $\bm P$. In practice, this doubles the number of equations in our modeling.

Remember that $p_I(C)$, where $C$ is a linear code (and so a vector space), is a homogeneous polynomial of degree $\dim(C)=k$ in the entries of its generator matrix $\bm G$. In fact, for $I\in \binom{[n]}{k}$, if $\bm B=(b_{ij})$ is the square $k\times k$ submatrix of $\bm G$ indexed by $I$, we have
\[
    p_I (C) = \det(\bm B) = \sum_{\sigma \in \mathcal{S}_k} \mathrm{sgn}(\sigma) \prod_{j=1}^n b_{\sigma(j)j}.
\]

Recall that, given a permutation matrix $\bm P$, the entries of $\bm G\bm P$ are bilinear in the entries of both $\bm P$ and $\bm G$, since
\[
(\bm G \bm P)_{ij} = \sum_{k=1}^n (\bm G)_{ik}\bm P_{kj}.
\]
\begin{remark}
    Any monomial matrix $\bm Q$ can be written as the product of a diagonal matrix and a permutation matrix, namely $\bm Q=\bm D\bm P$ and $\bm Q=\bm P\bm D'$. Since multiplication by a diagonal matrix does not affect the $\diagn$-equivalence class, the code generated by $\bm G\bm Q$ lies in the same equivalence class as the code generated by $\bm G\bm P$, therefore
\[
    [\bm G \bm Q]_{\diagn} = [\bm G \bm D \bm P]_{\diagn} = [\bm G \bm P \bm D']_{\diagn} = [\bm G \bm P]_{\diagn}.
\]
\end{remark}
By the theory developed in \Cref{sec:pluck_diag}, for each invariant in $\mathbb{F}_q\left(\Grkn \right)^\diagn$, we can write a polynomial equation whose solutions are the entries of $\bm P$. This is summarized in the following theorem.
\begin{theorem}\label{th:eqs}
    Let $\bm G_1$ and $\bm G_2$ be generator matrices of two equivalent codes on $\mathbb{F}_q$ of dimension $k$ and length $n$. Suppose that they are equivalent via the monomial matrix $\bm Q=\bm P \bm D$. Let $\mu \in \mathbb{F}_q\left(\Grkn \right)^\diagn$ such that $\mu$ is defined on the code generated by $\bm G_2$, then there exists a polynomial $h\in \mathbb{F}_q[x_{1,1},\dots,x_{n,n}]$ 
    which vanishes on the entries of the permutation matrix $\bm P$.
\end{theorem}
\begin{proof}
    Let $\mu\in \mathbb{F}_q\left(\Grkn \right)^\diagn$ such that $\mu$ is defined on the code generated by  $\bm G_2$. Without loss of generality we can identify codes by their generator matrices, and when we evaluate $\mu$ in a matrix $\bm G$, we are evaluating it in the code spanned by $\bm G$.
    Since there exist matrices $\bm S,\bm P,\bm D$ such that $\bm S \bm G_1 \bm P \bm D = \bm G_2$, we have
    \[
        [\bm G_2]_\diagn = [\bm S \bm G_1 \bm P \bm D]_\diagn = [\bm G_1 \bm P]_\diagn .
    \]
    Hence, because $\bm G_1 \bm P$ and $\bm G_2$ belong to the same equivalence class under $\diagn$, it follows that $\mu(\bm G_1 \bm P) = \mu(\bm G_2)$. The rational function $\mu$ is defined in terms of the Pl\"ucker coordinates: let $I\in \binom{[n]}{k}$, then
    \[
        p_I(\bm G \bm P) = \det \left( (\bm G \bm P)_I \right),
    \]
    where $(\bm G \bm P)_I$ is the square $k \times k$ submatrix of $\bm G \bm P$ whose columns are indexed by $I$. Hence,  $p_I(\bm G \bm P)$ is a homogeneous polynomial of degree $k$ in the entries of $(\bm G \bm P)_I$, which are in turn linear in the entries of $\bm P$. In short, each Pl\"ucker coordinate of $\bm G \bm P$ is a homogeneous polynomial of degree $k$ in the entries of $\bm P$.

    Since $\mu(\bm G_1 \bm P) = \mu (\bm G_2)$, if $\mu=\frac{g}{f}$, the polynomial $h\in \mathbb{F}_q[x_{1,1},\dots,x_{n,n}]$ is given by 
    \[
    h(\bm X)=g(\bm G_1 \bm X)-\mu(\bm G_2) f(\bm G_1 \bm X),
    \]
    where $\bm X=(x_{ij})$. By construction, $h$ vanishes when evaluated at $\bm P$.
\end{proof}
This means that, given two equivalent codes, for each invariant defined on one of them, we can construct a polynomial equation whose root is the permutation matrix that realizes the equivalence.

To find invariants, one can use Algorithm \ref{alg:gen}, but, for cryptographically relevant parameters, where $k = O(n)$, the matrix $\bm W_{k,n}$ has an exponential number of rows given by $\binom{n}{k}\simeq 2^n/\sqrt{n}$. Hence, an alternative method is needed. Observe that finding elements in the left kernel of $\bm W_{k,n}$ is feasible without writing them nor the matrix explicitly.  Indeed, let $I_1,J_1,I_2,J_2\in \binom{[n]}{k}$ be such that $I_1 \uplus J_1 =  I_2 \uplus J_2$. Then the vector $v \in \mathbb{Z}^{\binom{n}{k}}$ defined as
\[
    v_{I} = \begin{cases}
        1 & \text{if }I=I_1,J_1 ;\\
        -1 & \text{if }I=I_2,J_2 ;\\
        0 & \text{otherwise}
    \end{cases}
\]
lies in the left kernel of $\bm W_{k,n}$ and leads to the invariant $\frac{p_{I_1}p_{J_1}}{p_{I_2}p_{J_2}}$, which is computable in polynomial time, without writing down $v$.
\begin{lemma}
    \label{prop:IJ}
    Let $I_1,J_1,I_2,J_2\in \binom{[n]}{k}$ such that $I_1\uplus J_1 = I_2\uplus J_2$. Then $\frac{p_{I_1}p_{J_1}}{p_{I_2}p_{J_2}}$ is invariant under the action of $\diagn$.
\end{lemma}
\begin{proof}
    Let $C$ be in $\Grkn$ and let $\lambda= (\lambda_1,\dots,\lambda_n)\in\diagn$. Then $C$ and $\lambda \star C$ are in the same class under $\diagn$. The action of $\lambda$ on the Pl\"ucker is then given by
    \[
        p_I(\lambda \star C) = \left( \prod_{i\in I} \lambda_i\right) p_I(C).
    \]
    We apply this fact to the invariant $\mu=\frac{p_{I_1}p_{J_1}}{p_{I_2}p_{J_2}}$, obtaining 
    \begin{align*}
        \mu(\lambda \star C) & = \frac{p_{I_1}(\lambda \star C) p_{J_1}(\lambda \star C)}{p_{I_2}(\lambda \star C) p_{J_2}(\lambda \star C)} \\
       & = \frac{\left(\prod_{i_1\in I_1}\lambda_{i_1}\right)p_{I_1}(C) \left(\prod_{j_1\in J_1}\lambda_{j_1}\right) p_{J_1}(C)}{\left( \prod_{i_2\in I_2}\lambda_{i_2} \right) p_{I_2}(C) \left( \prod_{j_2\in J_2}\lambda_{j_2} \right) p_{J_2}(C)}  \\
       & = \frac{\left(\prod_{i\in I_1\uplus J_1}\lambda_{i}\right)p_{I_1}(C) p_{J_1}(C)}{\left( \prod_{i\in I_2\uplus J_2}\lambda_{i} \right) p_{I_2}(C) p_{J_2}(C)} \\
       & = \frac{p_{I_1}(C) p_{J_1}(C)}{p_{I_2}(C) p_{J_2}(C)}  = \mu(C)
    \end{align*}
    where the products of the $\lambda$'s cancel out since $I_1\uplus J_1 = I_2 \uplus J_2$.
\end{proof}

This method gives a polynomial-time procedure to obtain an invariant $\mu$ of degree 2 in the Pl\"uckers coordinates.

If we restrict the available invariants to the ones from \Cref{prop:IJ}, we reduce the problem of choosing subsets $I_1,J_1,I_2,J_2\in \binom{[n]}{k}$ such that $I_1\uplus J_1 = I_2\uplus J_2$ and $p_{I_2}(C_2)p_{J_2}(C_2)\ne 0$. For cryptographic relevant $n,k$ and $q$ and a random $C_2$, the condition $p_{I_2}(C_2)p_{J_2}(C_2)\ne 0$ holds with high probability since it defines a non-empty open set in the Zariski topology. 

We can summarize this fact in the following result.
\begin{theorem}\label{th:main}
    Let $\bm G_1$ and $\bm G_2$ be a random instance of the LCE problem, i.e. generator matrices of two equivalent codes over $\mathbb{F}_q$ of dimension $k$ and length $n$. Suppose that they are equivalent via the monomial matrix $\bm Q=\bm P \bm D$. 
    Then, with overwhelming probability over the choice of $\bm G_1$ and $\bm G_2$, there exists a polynomial $h\in \mathbb{F}_q[x_{1,1},\dots,x_{n,n}]$ of degree $2k$ with $2(k!)^2$ monomials which vanishes on the entries of the permutation matrix $\bm P$.
\end{theorem}
\begin{proof}
    For every $I_1,J_1,I_2,J_2\in \binom{[n]}{k}$ such that $I_1\uplus J_1 = I_2\uplus J_2$, \Cref{prop:IJ} ensure that the rational function $\mu = \frac{p_{I_1}p_{J_1}}{p_{I_2}p_{J_2}}$ is invariant under the action of $\diagn$. In order to apply \Cref{th:eqs}, we need to verify that $\mu$ is defined on $\bm G_2$. This is equivalent to the constraint $p_{I_2}(\bm G_2)p_{J_2}(\bm G_2)\ne 0$. This condition defines a non-empty Zariski-open subset of $\Gr_q(k,n)$. Hence, a random code $C$, i.e. a code generated by a uniformly random full-rank $k\times n$ matrix, satisfies this condition with overwhelming probability.

    Therefore, the polynomial which vanishes on $\bm P$ given by
    \[
        p_{I_1}(\bm G_1 \bm X)p_{J_1}(\bm G_1 \bm X)- \mu(\bm G_2) p_{I_2}(\bm G_1 \bm X)p_{J_2}(\bm G_1 \bm X) \in \mathbb{F}_q[x_{1,1},\dots,x_{n,n}]
    \]
    has $(k!)(k!) + (k!)(k!)=2(k!)^2$ monomials coming from $p_{I_i}(\bm G_1 \bm X)p_{J_i}(\bm G_1 \bm X)$ for $i=1,2$ and is of degree $2k$ in the entries of $\bm X$, since each $p_{I_i}(\bm G_1 \bm X)$ and $p_{J_i}(\bm G_1 \bm X)$ has degree $k$.
\end{proof}

\subsection{The Running Example on $\mathrm{Gr}_5(2,4)$}\label{subs:example_poly}

Recall the example from \Cref{sect:example24}. We obtained that, working on $\mathrm{Gr}(2,4)$, one algebraically independent invariant rational function is given by
\[
    \mu = \frac{p_{12}p_{34}}{p_{14}p_{23}} \in \K(\mathrm{Gr}(2,4))^{\dia_4}.
\]
Let $ \K = \mathbb{F}_5$ and let the $2\times 4 $ matrix
\[
    \bm G_1 = \begin{pmatrix}
        1 & 0 & 1 & 1 \\
        0 & 1 & 1 & 2
        \end{pmatrix}
\]
be the generator matrix of a linear code in $\mathrm{Gr}_5(2,4)$. 
Then, let $\bm Q = \bm D \bm P$ be a monomial matrix, and consider $\bm G_2$ to be the reduced row echelon form of $\bm G_1 \bm Q$
\[
\bm G_2 =
\begin{pmatrix}
1 & 0 & 1 & 2 \\
0 & 1 & 3 & 2
\end{pmatrix}.
\]
The Pl\"ucker coordinates of the code spanned by $\bm G_2$ are
\[
    p_{12} = 1,\, p_{13}=3,\, p_{14}=2,\, p_{23}=4,\, p_{24}=3,\, p_{34}=1,
\]
and hence, $\mu(\bm G_2)=1/3=2$. After writing $\mu = \frac{g}{f}$, with $g=p_{12}p_{34}$ and $f=p_{14}p_{23}$, from \Cref{th:eqs}, if $\bm X=(x_{ij})$ is an $4 \times 4$ matrix of unknowns, we can write the polynomial $h(\bm X) = g(\bm G_1 \bm X) - f(\bm G_1 \bm X)\mu(\bm G_2)$ which vanishes in $\bm P$. First of all, the matrix $\bm G_1 \bm X \in \mathbb{F}_5[x_{1,1},\dots,x_{n,n}]^{2 \times 4}$ is given by
\[
    \begin{pmatrix}
        x_{11} + x_{31} + x_{41} & & x_{12} + x_{32} + x_{42} & & x_{13} + x_{33} + x_{43} & & x_{14} + x_{34} + x_{44} \\
        \\
        x_{21} + x_{31} + 2x_{41} & & x_{22} + x_{32} + 2x_{42} & & x_{23} + x_{33} + 2x_{43} & & x_{24} + x_{34} + 2x_{44}
    \end{pmatrix}.
\]
Hence, $g(\bm G_1 \bm X)$ is given by $p_{12}(\bm G_1 \bm X)p_{34}(\bm G_1 \bm X)$, where
\[
    p_{12}(\bm G_1 \bm X) = \begin{vmatrix}
        x_{11} + x_{31} + x_{41} & & x_{12} + x_{32} + x_{42} \\
        \\
        x_{21} + x_{31} + 2x_{41} & & x_{22} + x_{32} + 2x_{42}
    \end{vmatrix},
\]
and
\[
    p_{34}(\bm G_1 \bm X) = \begin{vmatrix}
        x_{13} + x_{33} + x_{43} & & x_{14} + x_{34} + x_{44} \\
        \\
        x_{23} + x_{33} + 2x_{43} & & x_{24} + x_{34} + 2x_{44}
    \end{vmatrix}.
\]
Analogously, $f(\bm G_1 \bm X)$ is $p_{14}(\bm G_1 \bm X)p_{23}(\bm G_1 \bm X)$, with
\[
    p_{14}(\bm G_1 \bm X) = \begin{vmatrix}
        x_{11} + x_{31} + x_{41} & & x_{14} + x_{34} + x_{44} \\
        \\
        x_{21} + x_{31} + 2x_{41} & & x_{24} + x_{34} + 2x_{44}
    \end{vmatrix},
\]
while
\[
    p_{23}(\bm G_1 \bm X) = \begin{vmatrix}
        x_{12} + x_{32} + x_{42} & & x_{13} + x_{33} + x_{43}  \\
        \\
        x_{22} + x_{32} + 2x_{42} & & x_{23} + x_{33} + 2x_{43} 
    \end{vmatrix}.
\]

Observe that the polynomial $h(\bm X)$ has degree $4$ because both $g(\bm G_1 \bm X)$ and $f(\bm G_1 \bm X)$ are polynomials of degree $2k=4$ in the $x_{ij}$. One can also add the constraints
\[
    \begin{cases}
        \sum_{j=1}^n x_{ij} = 1 & \text{for }i=1,\dots,n; \\
        \sum_{i=1}^n x_{ij} = 1 & \text{for }j=1,\dots,n; \\
        x_{ij}^2 - x_{ij} = 0 & \text{for }i,j=1,\dots,n
    \end{cases}
\]
defining permutation matrices, as showed in \cite{saeed2017algebraic}. In this way, we see that the permutation matrix
\[
    \bm P =
    \begin{pmatrix}
    0 & 0 & 1 & 0 \\
    1 & 0 & 0 & 0 \\
    0 & 0 & 0 & 1 \\
    0 & 1 & 0 & 0
    \end{pmatrix}
\]
is a solution of $h(\bm X)$.

To add one more equation, since $\mu$ is also defined on $\bm G_1$, 
we can consider
\[
    h'(\bm X) = g(\bm G_2 \bm X^T) - f(\bm G_2 \bm X^T)\mu (\bm G_1),
\]
modeling the fact that $\bm P^{-1}= \bm P^T$ and that $\bm P^{-1}$ sends the equivalence class of $\bm G_2$ into the one of $\bm G_1$.
 Finally, to retrieve the matrix $\bm D$, one can use the canonical form from \cite{d2026group,cryptoeprint:2025/397}, which is
\[
    \bm D =
    \begin{pmatrix}
    1 & 0 & 0 & 0 \\
    0 & 3 & 0 & 0 \\
    0 & 0 & 4 & 0 \\
    0 & 0 & 0 & 2
    \end{pmatrix}.
\]

\begin{credits}
    \subsubsection{\ackname}
    The authors are members of the GNSAGA-INdAM  and CrypTO, the group of Cryptography and Number Theory of the Politecnico di Torino. This work was partially supported by the Italian Ministry of University and Research in the framework of the Call for Proposals for scrolling of final rankings of the PRIN 2022 call - Protocol no. 2022RFAZCJ. The authors would like to thank Michele Graffeo for its valuable comments, which helped to improve the overall quality of this work.
\end{credits}

\bibliographystyle{splncs04}
\bibliography{cryptobib/abbrev3,cryptobib/crypto, bibliography}

\end{document}